\documentclass[12pt]{amsart}

\usepackage{amsmath,amsthm,amsfonts,amssymb,latexsym}
\usepackage{verbatim,graphicx,ifthen,enumitem,color,a4}

\usepackage[T1]{fontenc}
\usepackage[applemac]{inputenc}

\usepackage[svgnames]{xcolor}
\usepackage{colortbl}
\usepackage[colorlinks=true,linkcolor=azul,citecolor=azul]{hyperref}

\usepackage{mathtools}

\theoremstyle{plain}
\newtheorem{theorem}{Theorem}[section]
\newtheorem{lemma}[theorem]{Lemma}
\newtheorem{corollary}[theorem]{Corollary}

\theoremstyle{definition}

\newtheorem{remark}[theorem]{Remark}

\newtheorem{definition}[theorem]{Definition}

\newtheorem{exa}[theorem]{Example}

\numberwithin{equation}{section}

\def \n #1.#2{\Vert #1\Vert_{#2}}

\definecolor{azul}{rgb}{0.1,0.6,0.86}
\definecolor{bluee}{rgb}{0,0.33,0.55}
\definecolor{naranja}{RGB}{249,153,96}

\def\noi{\noindent}

\def\bdem{\begin{proof}}
\def\edem{\end{proof}}

\def\bm{\left(\begin{array}}
\def\em{\end{array}\right)}
\def\ben{\begin{enumerate}}
\def\een{\end{enumerate}}
\def\barr{\begin{array}}
\def\earr{\end{array}}
 \def\bit{\begin{itemize}}
\def\eit{\end{itemize}}
\def\beq{\begin{equation}}
\def\eeq{\end{equation}}
\def\bdes{\begin{description}}
\def\edes{\end{description}}

\def\fii{\varphi }

\def\R{\mathbb{R}}

\newcommand{\peso}[1]{ \quad \mbox{  #1 } \quad }

\newcommand{\norm}[1]{\left\lVert#1\right\rVert}
\newcommand{\parenthesis}[1]{\left( #1 \right)}

\def\bdem{\begin{proof}}
\def\edem{\end{proof}}

\begin{document}

\title{Revisiting Yano and Zygmund extrapolation theory} 
\author{Elona Agora, Jorge Antezana, Sergi Baena-Miret, Mar\'{\i}a J. Carro} 

\address{E. Agora, Instituto Argentino de Matem\'atica ``Alberto P. Calder\'on'', 
1083 Buenos Aires, Argentina.}
\email{elona.agora@gmail.com}

\address{J. Antezana, Department of Mathematics, 
Faculty of Exact Sciences, 
National University of La Plata, 
1900 La Plata, Argentina,  \and
Instituto Argentino de Matem\'atica ``Alberto P. Calder\'on'', 
1083 Buenos Aires, Argentina.}
\email{antezana@mate.unlp.edu.ar}

\address{Sergi Baena-Miret, Department of Mathematics and Informatic, University of Barce- lona, Barcelona, Spain.}
\email{sergibaena@ub.edu}

\address{M. J. Carro, Department of Mathematical Analysis and Applied Mathematics,
Complu- tense University of Madrid, Madrid, Spain.}
\email{mjcarro@ucm.es}

\subjclass[2010]{46E30, 28A10, 47A30}
\keywords{Yano's Extrapolation theory, Zygmund's extrapolation theory, Calder\'on type operators, decreasing rearrangement estimates}
\thanks{This work has been partially supported by Grants MTM2016-75196-P and MTM2017-83499-P (MINECO / FEDER, UE),  PIP-152 (CONICET), PICT 2015-1505 (ANPCYT), and 11X829 (UNLP)}

\begin{abstract} 
We prove a pointwise estimate for the decreasing rearrangement of $Tf$, where $T$ is any sublinear operator satisfying the weak-type boundedness 
$$
T:L^{p,1}(\mu) \to L^{p,\infty}(\nu), \quad  \forall p: 1<p_0 < p\leq p_1<\infty,
$$
with norm controlled by  $C\varphi\left(\left[{p_0^{-1}} - p^{-1}\right]^{-1}\right)$ and $\varphi$ satisfies some admissibility conditions.
The pointwise estimate is:
\begin{equation*}
\begin{split}
    (Tf)^*_\nu(t) &\lesssim \frac 1{p_0 - 1}\left(\frac 1{t^\frac 1{p_0}}\int_0^t \varphi\left(1 - \log \frac rt\right)f^*_\mu(r)\frac{dr}{r^{1 - \frac 1{p_0}}} + \frac 1{t^\frac 1{p_1}}\int_t^\infty f^*_\mu(r)\frac{dr}{r^{1 - \frac 1{p_1}}}\right).
    \end{split}
\end{equation*}
In particular, this estimate allows to obtain extensions of Yano and Zygmund extrapolation results. 

\end{abstract}

\date{\today}

\maketitle

\pagestyle{headings}\pagenumbering{arabic}\thispagestyle{plain}

\markboth{Revisiting Yano and Zygmund extrapolation theory}
{E. Agora, J. Antezana, S. Baena-Miret and M.J. Carro}


\section{Introduction}

In 1951, Yano (see \cite{y:y,z:z}) using the ideas of Titchmarsh
in \cite{t:t}, proved that for every sublinear operator $T$ satisfying that, for some $\alpha>0$ and for every $1<p\le p_1$,
$$
T:L^p(\mu) \longrightarrow L^p(\nu), \qquad \frac C{(p-1)^\alpha},
$$
where $(\mathcal M, \mu)$ and $(\mathcal N, \nu)$ are two finite measure spaces, it holds  that
$$
T:L(\log L)^\alpha(\mu)
\longrightarrow L^1(\nu), 
$$
where the space $L(\log L)^\alpha(\mu) $ is defined as the set of $\mu$-measurable functions such that $$||f||_{L(\log L)^\alpha(\mu)} = \int_0^\infty f_\mu^*(t)\left(1 + \log^+ \frac 1t\right)^\alpha \, dt < \infty,$$ where $a^+ = \max(a,0)$, for every $a \in \mathbb R$, $f^*_\mu$  is  the decreasing rearrangement of $f$ with respect to the measure $\mu$ defined as
$$
f^*_\mu(t)=\inf\Big\{y>0: \lambda_f^\mu(y)\le  t\Big\}, \qquad t > 0,
$$
and 
$$
\lambda_f^\mu(y)=\mu\big(\big\{x\in\mathbb R^n: |f(x)|>y\big\}\big),\qquad y > 0,
$$ 
is the distribution function of $f$ with respect to $\mu$. (Here we are using the standard notation $\mu(E)=\int_E \, d\mu(x)$ for every $\mu$-measurable set $E \subseteq X$. If $d\mu=dx$, we shall write $f^*$, $\lambda_f(y)$ and  $|E|$. See \cite{bs:bs} for more details about this topic).

If the measures involved are not finite, it was proved in \cite{c:c} and \cite{c1:c1} that
under a weaker condition on the operator $T$, namely
$$
\bigg(\int_{\mathcal N} |T\chi_A(x)|^p\, d\nu(x)\bigg)^{1/p}\le \frac
C{p-1}\mu(A)^{1/p},
$$
for every $\mu$-measurable set $A\subset\mathcal M$ and every $1<p\le p_0$, with $C$
independent of $A$ and
$p$, then
$$
T:L(\log L)^\alpha(\mu) \longrightarrow M(\phi),
$$
where  $M(\phi)$ is the maximal Lorentz space associated
to the function
$\phi(t)=\frac t{1+\log^+ t}$, $t > 0$; that is,
$$
\n f.{M(\phi)}=\sup_{t>0}\phi(t)f^{**}_{\nu}(t) = \sup_{t>1}\frac {t f^{**}_{\nu}(t)}{1+\log t},
$$
where $f^{**}_\nu(t) = \frac 1t\int_0^t f^*_\nu(s)\,ds$, $t > 0$. These results belong to what is known as Yano's extrapolation theory. 

On the other hand, in \cite[p.  119]{z:z} it was seen that if $T$ is a linear
operator satisfying
\begin{equation}
\n {Tf}.{L^p(\nu)} \le C p \n {f}.{L^p(\mu)},\label {zz}
\end{equation}
for  every $p$ near $\infty$ and for $\mu$ and $\nu$ being finite measures,
then
$$
T:L^\infty(\mu)\longrightarrow L_{\exp}(\nu), 
$$
where $L_{\exp}(\nu)$ is the set of $\nu$-measurable functions satisfying that $$||f||_{L_{\exp}(\nu)} = \sup_{0 < t < 1} \frac{f_\nu^{**}(t)}{1 + \log \frac 1t},$$
and this result was also extended to the case of general measures (see \cite{c1:c1}) proving that, if $T$ is a linear operator satisfying
(\ref {zz}), then
\begin{equation}\label{zzz}
\sup_{0 < t < 1} \frac{ (Tf)^{**}_{\nu}(t) }{ 1+\log\frac 1t}
\lesssim ||f||_{L^\infty(\mu)} + \int_1^\infty f^{**}_{\mu}(s) \frac {ds}s,
\end{equation}
where,  as usual, we write $A \lesssim B$ if there exists a positive constant $C>0$, independent of $A$ and $B$, such that $A\leq C B$. If $A\lesssim B$ and $B\lesssim A$, then we write $A\approx B$. 
\noindent These results belong to what is known as Zygmunds's extrapolation theory. 

\

For the purpose of this work, it is also interesting to mention that analogue results are known to the ones commented above in the case that, for some $1 < p_0<p_1 < \infty$ and every $p_0<p<p_1$,
$$
T:L^p(\mu) \longrightarrow L^p(\nu)
$$
with
$$
||T||_{L^p(\mu) \longrightarrow L^p(\nu)}\lesssim  \frac 1{(p-p_0)^\alpha} \quad\mbox{or}\quad  ||T||_{L^p(\mu) \longrightarrow L^p(\nu)} \lesssim\frac 1{(p_1-p)^\alpha}, \qquad \alpha > 0.
$$

The exact statements are the following:

\begin{theorem} (\cite{cm:cm}) 
Let $T$ be a sublinear operator satisfying that for every $p_0 < p \leq p_1$,
$$
T:L^{p}(\mu)\longrightarrow L^{p}(\nu), \qquad \frac C{p-p_0}.
$$
Then,
\begin{eqnarray*}
\sup_{t>1} 
\frac{\Big(\displaystyle
\int_0^{t} (Tf)^*_{\nu}(s)^{p_0}\,ds\Big)^{1/p_0}}
{1+\log t}\lesssim
\|f\|_{L^{p_0}(\mu)}+\int_0^1
\frac{\Big(\displaystyle
\int_0^r f^*_{\mu}(s) ^{p_0}\,ds\Big)^{1/p_0}}r 
\,dr.
\end{eqnarray*}
\end{theorem}

\begin{theorem} (\cite{cm:cm}) 
Let $T$ be a sublinear operator satisfying that for every $p_0\le p <p_1$, 
$$
T:L^{p}(\mu)\longrightarrow L^{p}(\nu), \qquad \frac C{p_1-p}. 
$$
 Then,
\begin{eqnarray*}
\sup_{0 < t < 1}  
\frac{\Big(\displaystyle
\int_{t}^\infty  (Tf)^*_\nu(s) ^{p_1}\,ds\Big)^{1/p_1}}
{1+\log\frac 1t}\lesssim
\|f\|_{L^{p_1}(\mu)}+\int_1^\infty
\frac{\Big(\displaystyle
\int_r^{\infty}  f^*_\mu(s) ^{p_1}\,ds\Big)^{1/p_1}}r
\,dr.
\end{eqnarray*}
\end{theorem}

\ 

\medskip
There is also a Yano's extrapolation theorem concerning weak-type spaces. In 1996, N.Yu. Antonov \cite{a:a2} proved that there is almost everywhere convergence for the
Fourier series of every function in $L\log L\log_3 L(\mathbb T)$, where $\mathbb T$ represents the unit circle and, for an arbitrary measure $\mu$,  $$||f||_{L\log L\log_3 L(\mu)} = \int_0^\infty f^*_\mu(t) \log_1 \frac 1t \log_3 \frac 1t\, dt < \infty,$$ with \begin{equation}\label{eq:log_def}
    \log_1 t = 1 + \log^+ t \quad \text{ and } \quad \log_k t = \log_1 \log_{k - 1} t \, \, \text{ for } k > 1, \qquad \qquad t > 0.
\end{equation} Indeed, even though he did not write it explicitly, behind his ideas there is an extrapolation argument (see \cite{a:a,c:c} for more details). Before we make its statement precise, let us recall that, given $1 \leq p < \infty$ and $0 < q < \infty$, the Lorentz spaces $L^{p,q}(\mu)$ are defined as the set of $\mu$-measurable functions $f$ such that $$\Vert f\Vert _{L^{p,q}(\mu)}= \left(\int_0^\infty t^{\frac qp -1} f^*_\mu(t)^q \, dt\right)^{1/q}  =  \left(p \int_0^\infty t^{q-1} \lambda_f^\mu(y)^\frac qp\,dy\right)^{1/q}  <\infty,$$ and $$
\Vert f\Vert _{L^{p,\infty}(\mu)}= \sup_{t>0} t^{\frac 1p} f^*_\mu(t) =\sup_{y>0} y \lambda_{f}^\mu(y)^\frac 1p <\infty.$$

\begin{theorem}\label{thrm:restricted_weak_type_Yano}   If $T$ is a sublinear operator such that for some $\alpha > 0$ and for every $1 < p \leq p_0$, 
$$
T:L^{p,1}(\mu) \rightarrow L^{p,\infty}(\nu), \qquad \frac C{(p - 1)^\alpha},
$$
then 
$$
T:L(\log L )^\alpha\log_3 L(\mu) \rightarrow L^{1,\infty}(\nu).
$$ 
\end{theorem}

Since all the spaces above mentioned are rearrangement invariant, all the results could be also obtained if we find a good estimate for the function $(Tf)^*$ and, moreover, in this case, we can deduce  boundedness of $T$ in other rearrangement invariant spaces as well. At this point, we should recall  the following  well-known pointwise estimate: 

\begin{theorem}{\upshape (\cite[Ch. 4.4 Theorem 4.11]{bs:bs})}\label{classical}
Let $1 \leq p_0 < p_1 < \infty$. A sublinear operator $T$ satisfies that 
$$
T:L^{p_0,1}(\mu) \rightarrow L^{p_0,\infty} (\nu)\qquad \text{ and } \qquad T:L^{p_1,1}(\mu) \rightarrow L^{p_1,\infty}(\nu),
$$ if and only if, for every $t > 0$ and every $\mu$-measurable function $f$,
$$(Tf)^*_\nu(t) \lesssim \frac 1{t^\frac 1{p_0}} \int_0^t f^*_\mu(s)\,\frac{ds}{s^{1 - \frac 1 {p_0}}} + \frac 1{t^\frac 1{p_1}} \int_t^\infty f^*_\mu(s) \frac{ds}{s^{1 - \frac 1 {p_1}}}.$$
\end{theorem}

Moreover, with the goal of finding interesting pointwise estimates for the function $(Tf)^*$ under weaker conditions on $T$,   the following results have been recently proved in \cite{aabc:aabc} (see Definition \ref{admissible} for the notion of admissible function):

\begin{theorem} \label{paper31} Let $T$ be a sublinear operator and $\varphi$ some admissible function. For every $1\le p<\infty$,
$$
T:L^{p,1}(\mu) \rightarrow L^{p, \infty}(\nu), \qquad  C\varphi(p), 
$$
if and only if, for every $t > 0$ and every $\mu$-measurable function $f$,
\begin{align*}
(Tf)^*_\nu(t)&\lesssim \frac 1t \int_0^t f^*_\mu(s) \, ds + \int_t^\infty \left(1 + \log\frac st\right)^{-1}\fii\left(1 + \log\frac st\right)  f^*_\mu(s) \, \frac{ds}s.
\end{align*}
\end{theorem}

\medskip

We observe that, if $\varphi(p)=p$ (which is an admissible function), 
$$
(Tf)^{**}_\nu(t) \lesssim\frac 1t \int_0^t  f^{**}_\mu(s) ds + \int_t^\infty f^*_\mu(s) \frac{ds }s
$$
and hence
\begin{eqnarray*}
\sup_ {0 < t < 1} \frac{ (Tf)^{**}_{\nu}(t) }{ 1+\log\frac 1t}
&\lesssim&  \n f.{\infty} +  \int_1^\infty f^{*}_{\mu}(s) \frac {ds}s +  \sup_ {0<t<1} \frac 1 { 1+\log\frac 1t} \int_t^1 f^{*}_{\mu}(s) \frac {ds}s 
\\
&\lesssim&
 \n f.{\infty} +  \int_1^\infty f^{*}_{\mu}(s) \frac {ds}s, 
\end{eqnarray*}
and we recover \eqref{zzz}.  

We have to mention here that although we have obtained a better estimate, we have, initially, assumed a stronger condition, since our boundedness hypothesis also includes the case $p=1$. However, if the boundedness information is only for $p\ge p_0>1$, we also have the following result  (see \cite{aabc:aabc}) from which \eqref{zzz} can also be recovered:

\begin{theorem} \label{paper32} Take $1 \leq p_0 < p_1 \leq \infty$ and let $T$ be a sublinear operator and $\varphi$ some admissible function. Assume that for every $p_0 \leq p < p_1$,
$$
T:L^{p,1}(\mu) \rightarrow L^{p, \infty}(\nu), \qquad C\varphi\left(\left[\frac 1p - \frac 1{p_1}\right]^{-1}\right).
$$
Then, for every $t > 0$ and every $\mu$-measurable function $f$:
\begin{enumerate}
    \item[(i)] If $p_1 < \infty$, 
    \begin{equation}\label{eq:theorem_paper32_1}
        (Tf)_\nu^*(t)\lesssim \frac 1{t^{\frac 1{p_0}}}\int_0^tf^*_\mu(s)\,\frac{ds}{s^{1 - \frac 1{p_0}}}+ \frac1{t^{\frac 1{p_1}}}\int_t^\infty \fii\left(1+\log \frac st\right) f^*_\mu(s)\,\frac{ds}{s^{1 - \frac 1{p_1}}}.
    \end{equation}

    \item[(ii)] If $p_1 = \infty$, 
    \begin{equation}\label{eq:theorem_paper32_2}
        (Tf)_\nu^*(t)\lesssim \frac 1{t^{\frac 1{p_0}}}\int_0^tf^*_\mu(s)\,\frac{ds}{s^{1 - \frac 1{p_0}}}+  \int_t^\infty \left(1+\log \frac st\right)^{-1}\fii\left(1+\log \frac st\right) f^*_\mu(s)\,\frac{ds}s.
    \end{equation}
\end{enumerate}
Conversely, if \eqref{eq:theorem_paper32_1} holds then, for every $p_0 \leq p < p_1$, $$
||T||_{L^{p,1}(\mu) \rightarrow L^{p, \infty}(\nu)} \lesssim 
  \left[\frac 1p - \frac 1{p_1}\right]^{-1} \varphi\left(\left[\frac 1p - \frac 1{p_1}\right]^{-1}\right),
$$ while if \eqref{eq:theorem_paper32_2} holds, $
||T||_{L^{p,1}(\mu) \rightarrow L^{p, \infty}(\nu)} \lesssim \varphi(p)$.
\end{theorem}

\bigskip

Our goal in this note is to prove similar results to those in Theorems \ref{paper31} and  \ref{paper32}, to obtain extensions of Yano's extrapolation results. Moreover, we want to emphasize here, that we obtain  stronger results with a simpler proof because 
contrary to what happens in the proof of the above mentioned results of Yano and Zygmund, where the function $f$ is decomposed in an infinite sum of functions $f_n$, our proof follows the idea of Theorem \ref{classical} where the function $f$ is decomposed as the sum of just two functions. 
\bigskip

The paper is organized as follows. In Section \ref{sec:def_technic},  we present previous results, the necessary definitions and some technical lemmas which shall be used later on, and Section  \ref{sec:main} contains our main results.

\section{Definitions, previous results and lemmas}\label{sec:def_technic}

\subsection{Admissible functions}

\begin{definition}[\cite{aabc:aabc}] \label{admissible}
A function $\varphi:[1, \infty]\to [1, \infty]$ is said to be \textit{admissible} if satisfies the following conditions:
\begin{enumerate}
\item[a)] $\fii(1) = 1$ 
and $\fii$ is log-concave, that is 
 $$
 \theta\log \varphi(x)+(1-\theta) \log \varphi(y) \le \log \varphi(\theta x+(1-\theta)y), \qquad \forall x, y \geq 1, \ 0\le \theta\le 1. 
$$

\item[b)] There exist $\gamma,\beta>0$ such that for every $x \geq 1$,
\begin{equation}\label{la espada y la pared}
\frac\gamma x\leq \frac{\fii'(x)}{\fii(x)}\leq \frac\beta x.
\end{equation}
\end{enumerate}
\end{definition}

\noindent Observe that \eqref{la espada y la pared}  implies that $\fii$ is increasing, as well as that 
$$
x^\gamma\leq \fii(x)\leq x^\beta. 
$$

\noindent Besides, since for every $x,y \geq 1$,
\begin{equation*}
\begin{split}
 \log\fii(xy) &=\int_1^y(\log \fii)'(s)\,ds+\int_y^{xy}(\log \fii)'(s)ds \leq \log \fii(y) + \beta \log x,   
\end{split}
\end{equation*}

\noindent it also holds that \begin{equation}\label{por la logconcavidad2}
    \varphi(xy) \leq x^\beta \varphi(y).
\end{equation}

\medskip

\begin{exa} 
Given $m \in \mathbb N$ and using the notation in \eqref{eq:log_def}, if $\gamma>0$ and $\beta_1,\ldots,\beta_m\geq 0$, the function 
$$
\fii(x)=x^\gamma \prod_{k=1}^m \big(\log_k x\big)^{\beta_k}, \qquad x \geq 1,
$$
is admissible.
\end{exa}

The next lemma is a simple computation for admissible  functions which shall be fundamental in the proof of our main results. 

\begin{lemma}\label{L infimo log log}
Let $\fii$ be an admissible function. For $x\in\R$ and $1 \leq q_0 < \infty$,
$$
\inf_{q \in [q_0, \infty)}\fii(q)e^{-\frac xq }\leq 
\begin{cases}
\fii(q_0) e^{-\frac x{q_0}}, & \mbox{if $x\geq 0$},\\
q_0^\beta e^{\frac 1{q_0}}\fii\left(1-x\right), & \mbox{if $x < 0$}.
\end{cases}
$$
\end{lemma}
\bdem
If $x \ge 0$, the infimum is attained at $q = q_0$, and if $x<0$, we take $q = q_0(1-x)$ and make use of \eqref{por la logconcavidad2}. 
\edem

\subsection{Calder\'on type operators}

\begin{definition}
\noindent Let $1 \leq p_0, p_1 \leq \infty$ and let  $\fii$ be an admissible  function. Then, for every positive and  real valued measurable function $f$ and $t > 0$, we define 
\begin{equation*}
\begin{split}
    P_{p_0, \varphi}f(t)&:= \frac 1{t^{\frac 1{p_0}}}\int_0^t\varphi\left( 1 - \log\frac st \right)f(s)\,\frac{ds}{s^{1 - \frac 1{p_0}}}, \\ Q_{p_1}f(t)&:=  \frac1{t^{\frac 1{p_1}}}\int_t^\infty  f(s)\,\frac{ds}{s^{1 - \frac 1{p_1}}},
\end{split}
\end{equation*}
and
$$
R_{p_0,p_1,\fii}f(t) := P_{p_0, \fii}f(t)+Q_{p_1}f(t).
$$
\end{definition}

\medskip

\noi In particular, if  $p_0 = 1$, $p_1=\infty$, and $\fii(x) = 1$, we recover the Calder\'on operator  \cite{bs:bs}
$$
Rf(t) := Pf(t) + Qf(t), \qquad t > 0,
$$ 
where $P$ and $Q$ are respectively the Hardy operator and its conjugate
$$
Pf(t) = \frac 1t\int_0^tf(s)\,ds, \qquad Qf(t) =\int_t^\infty f(s)\,\frac{ds}s, \qquad t > 0. 
$$
We observe that, in general,
\begin{equation}\label{stst}
R_{p_0, p_1, \varphi}f(t) = \int_0^1 \varphi\left( 1 - \log s \right) f(st)\,\frac{ds}{s^{1 - \frac 1{p_0}}} + \int_1^\infty f(st)\,\frac{ds}{s^{1 - \frac {1}{p_1}}}, \qquad t > 0.
\end{equation}

\medskip

\begin{lemma}\label{GeneralizedCalderonProperty} 
Let $1 \leq p_0, p_1 \le \infty$. For an arbitrary measure $\mu$ and every $\mu$-measurable function $f$,
 $$
 R_{p_0,p_1,\fii}(f^*_\mu)^{**}(t) = R_{p_0,p_1,\fii}(f^{**}_\mu)(t), \qquad t > 0.
 $$
\end{lemma}

\begin{proof}  By \eqref{stst}, clearly,  $R_{p_0,p_1,\fii}(f^*_\mu)$ is a decreasing function. Then,  it holds that    
$$
R_{p_0,p_1,\fii}(f^*_\mu)^{**}(t)= P \big( R_{p_0,p_1,\fii}(f^*_\mu) )(t), \qquad t > 0, 
$$  
and the result follows immediately by the Fubini's theorem. 
\end{proof}

\medskip

\section{Main Results}\label{sec:main}

Throughout this section, if not specified, $(\mathcal M, \mu)$ and $(\mathcal N, \nu)$ are two arbitrary measure spaces.

\begin{theorem}\label{thrm:main_result1}
Take $1 < p_0 < p_1 < \infty$ and let $T$ be a sublinear operator and $\varphi$ some admissible function. If for every $p_0 < p \leq p_1$,
\begin{equation}\label{main_theorem:hypothesis_eq}
T:L^{p, 1}(\mu)\longrightarrow  L^{p, \infty} (\nu), \qquad C\varphi\left(\left[\frac 1{p_0} - \frac 1p\right]^{-1}\right), 
\end{equation}
then, for every $t>0$ and every $\mu$-measurable function $f$, 
\begin{equation}\label{main_theorem:hypothesis_eq2}
\begin{split}
    (Tf)^*_\nu(t) &\lesssim \frac 1{p_0 - 1}\left(\frac 1{t^\frac 1{p_0}}\int_0^t \varphi\left(1 - \log \frac rt\right)f^*_\mu(r)\frac{dr}{r^{1 - \frac 1{p_0}}} + \frac 1{t^\frac 1{p_1}}\int_t^\infty f^*_\mu(r)\frac{dr}{r^{1 - \frac 1{p_1}}}\right) \\ &= \frac 1{p_0 - 1}R_{p_0, p_1, \fii}(f^*_\mu)(t).
    \end{split}
\end{equation}

\noindent Conversely, if $(Tf)_{\nu}^*(t) \lesssim R_{p_0, p_1, \fii}(f^*_\mu)(t)$ for every $t > 0$, then, for each $p_0 < p \leq p_1$, \eqref{main_theorem:hypothesis_eq} holds.


\end{theorem}

\begin{proof}

First assume that \eqref{main_theorem:hypothesis_eq} applies for every $p_0 < p \leq p_1$. Then, if $f = \chi_E$ for some $\mu$-measurable set $E \subseteq \R^n$ such that $\mu(E) < \infty$, for every $t > 0$,
\begin{align*}
    R_{p_0, p_1, \fii}(f^*_\mu)(t) &= \frac 1{t^\frac 1{p_0}}\int_0^t \varphi\left(1 - \log \frac rt\right)\chi_{(0,\mu(E))}(r)\frac{dr}{r^{1 - \frac 1{p_0}}} + \frac 1{t^\frac 1{p_1}}\int_t^\infty \chi_{(0,\mu(E))}(r)\frac{dr}{r^{1 - \frac 1{p_1}}} \\ &= \left( \frac 1{t^\frac 1{p_0}}\int_0^t \varphi\left(1 - \log \frac rt\right)\frac{dr}{r^{1 - \frac 1{p_0}}} + \frac 1{t^\frac 1{p_1}}\int_t^{\mu(E)} \frac{dr}{r^{1 - \frac 1{p_1}}} \right)\chi_{(0,\mu(E))}(t) \\& + \left( \frac 1{t^\frac 1{p_0}}\int_0^{\mu(E)} \varphi\left(1 - \log \frac rt\right)\frac{dr}{r^{1 - \frac 1{p_0}}} \right) \chi_{(\mu(E), \infty)}(t) \\ & \geq p_0\left[ \left(\frac {\mu(E)}t\right)^{\frac1{p_1}}\chi_{(0,\mu(E))}(t) + \varphi\left(1 - \log \frac {\mu(E)}t\right)\left(\frac {\mu(E)}t\right)^{\frac1{p_0}} \chi_{(\mu(E),\infty)}(t)\right],
\end{align*}
where in the last estimate we have used that $p_1 > p_0$, $\varphi(1) = 1$ and that $\varphi\left(1 - \log \frac s t\right)$ is a decreasing function on $s \in (0,t)$. 

Hence, since by hypothesis, for every $p_0 < p \leq p_1$  we have that
$$ (T\chi_E)^*_\nu(t) \leq C\varphi\left(\left[\frac 1{p_0} - \frac 1p\right]^{-1}\right)\left(\frac{\mu(E)}{t}\right)^\frac 1p = C \left[\varphi\parenthesis{q}\left(\frac{\mu(E)}{t}\right)^{-\frac 1q}\right]\left(\frac{\mu(E)}{t}\right)^{\frac1{p_0}}, \qquad t > 0, 
$$
with $\frac 1q = \frac 1{p_0} - \frac 1p$, the result for characteristic functions plainly follows by taking the infimum for $q \in \left[\frac{p_1 p_0}{p_1 - p_0}, \infty\right)$ (see Lemma \ref{L infimo log log}) since then, for every $ t > 0$,
\begin{align*}
    (T\chi_E)^*_\nu\left(t\right) &\lesssim \left( \frac{\mu(E)}t \right)^{\frac 1{p_1}}\chi_{\left(0, \mu(E)\right)}(t) + \varphi\left(1 - \log \frac{\mu(E)}t \right)\left( \frac{\mu(E)}t \right)^{\frac 1{p_0}}\chi_{\left(\mu(E),\infty\right)}(t)
    \\
    & \lesssim R_{p_0,p_1,\fii}((\chi_E)_\mu^*)(t).
\end{align*}

The extension to simple functions with sets of finite measure with respect to $\mu$ follows the same lines as the proof of Theorem III.4.7 of  \cite{bs:bs}. We include the computations adapted to our case for the convenience of the reader. First of all, consider a positive simple function 
\begin{equation*}
f=\sum_{j=1}^na_j\chi_{F_j},
\end{equation*}
where $F_1\subseteq F_2\subseteq \ldots\subseteq F_n$  have finite measure with respect to $\mu$. Then
$$
f^*_\mu=\sum_{j=1}^na_j\chi_{[0,\mu(F_j))}.
$$
Using what we have already proved for characteristic functions we get that for every $t > 0$,
\begin{equation*}
\begin{split}
(Tf)^{**}_\nu(t)&\lesssim \sum_{j=1}^n a_j \left( T(\chi_{F_j})\right)^{**}_\nu(t)\lesssim \sum_{j=1}^n a_j \left( R_{p_0,p_1,\fii}(\chi_{[0,\mu(F_j))})\right)^{**}(t)\\
&=\left( R_{p_0,p_1,\fii}\left(\sum_{j=1}^n a_j \chi_{[0,\mu(F_j)})\right)\right)^{**}(t)=  R_{p_0,p_1,\fii}(f^*_\mu)^{**}(t).
\end{split}
\end{equation*}
Further, since $ R_{p_0,p_1,\fii}(f^*_\mu)^{**}= R_{p_0,p_1,\fii}(f^{**}_\mu)$ (see Lemma~\ref{GeneralizedCalderonProperty}) we obtain that
\begin{equation}\label{T vs S doble star}
(Tf)^{**}_\nu(t)\lesssim  R_{p_0,p_1,\fii}(f^{**}_\mu)(t), \qquad t > 0.
\end{equation}

Now fix $t>0$ and consider the set $E=\{x: f(x)>f^*_\mu(t)\}$. Using this set define
\begin{equation*}
g=(f-f^*_\mu(t))\chi_E \peso{and} h=f^*_\mu(t)\chi_E + f\chi_{E^c},
\end{equation*}
so that $f = g + h$ and
$$
g^*_\mu(s)=(f^*_\mu(s)-f^*_\mu(t))\chi_{(0,t)}(s)  \peso{and}  h^*_\mu(s)=\min\{f^*_\mu(s),f^*_\mu(t)\}, \qquad s > 0.
$$
Since \eqref{main_theorem:hypothesis_eq} holds with $p = p_1$, the corresponding weak inequality leads to
\begin{align*}
    (Th)^*_\nu(t/2)&\lesssim \frac1{t^{\frac 1{p_1}}}\int_0^\infty h^*_\mu(s)\,\frac{ds}{s^{1-\frac 1{p_1}}} \lesssim f^*_\mu(t) + \frac1{t^{\frac 1{p_1}}}\int_t^\infty f^*_\mu(s)\,\frac{ds}{s^{1-\frac 1{p_1}}} \\ & \leq P_{p_0,\fii}(f^*_\mu)(t) + Q_{p_1}(f^*_\mu)(t),    
\end{align*} 
where we have used that $f^*_\mu(t) \leq P_{p_0,\fii}(f^*_\mu)(t)$. 

On the other hand, on account of \eqref{T vs S doble star} we get 
\begin{equation}\label{eq:Tg**main_theorem}
\begin{split}
    (Tg)^{**}_\nu(t)&\lesssim R_{p_0,p_1,\fii}(g^{**}_\mu)(t) = P_{p_0,\fii}(g^{**}_\mu)(t)+Q_{p_1}(g^{**}_\mu)(t)
\end{split}
\end{equation}

\noindent and for the first term of the right hand side of \eqref{eq:Tg**main_theorem} we deduce that 

\begin{equation}\label{eq:Pg**mu}
\begin{split}
     P_{p_0,\fii}(g^{**}_\mu)(t) &= \frac 1{t^\frac 1{p_0}}\int_0^t \fii\left( 1 - \log\frac st \right) \frac 1s\int_0^s g^*_\mu(r)\,dr\frac{ds}{s^{1 - \frac 1{p_0}}} \\ &\leq \frac 1{t^\frac 1{p_0}}\int_0^t \fii\left( 1 - \log\frac st \right) \int_0^s f^*_\mu(r)\,dr\frac{ds}{s^{2 - \frac 1{p_0}}} \\ &= \frac 1{t^\frac 1{p_0}}\int_0^t f^*_\mu(r) \int_r^t \fii\left( 1 - \log\frac st \right)\frac{ds}{s^{2 - \frac 1{p_0}}} \,dr \\ &\leq \frac{p_0}{p_0 - 1}\left(\frac 1{t^\frac 1{p_0}}\int_0^t  \fii\left( 1 - \log\frac rt \right) f^*_\mu(r) \frac{dr}{r^{1 - \frac 1{p_0}}}\right) \\ &= \frac{p_0}{p_0 - 1} P_{p_0,\fii}(f^*_\mu)(t),
\end{split}
\end{equation}

\noindent while for the second term

\begin{align*}
    Q_{p_1}(g^{**}_\mu)(t) = \frac1{t^{\frac 1{p_1}}}\int_t^\infty \frac 1s\int_0^s g^*_\mu(r)\,dr\,\frac{ds}{s^{1-\frac 1{p_1}}} \leq \frac{p_1}{p_1-1}f^{**}_\mu(t) \leq \frac{p_0}{p_0-1} P_{p_0,\fii}(f^*_\mu)(t).
\end{align*}



Thus,
\begin{align*}
(Tf)^*_\nu(t)& \leq  2(Tg)^{**}_\nu(t) + (Th)^*_\nu(t/2) \lesssim \frac 1{p_0 - 1} R_{p_0,p_1,\varphi}(f^*_\mu)(t),
\end{align*}

\noindent and the general case follows from the density of the simple functions to any $\mu$-measurable function and dividing a $\mu$-measurable function in its positive and negative parts.

\

Conversely, assume that $(Tf)^*_\nu(t) \lesssim R_{p_0,p_1,\fii}(f^*_\mu)(t)$ for every $t > 0$ and fix some $p \in (p_0, p_1]$. The operator $R_{p_0,p_1,\fii}$ is a kernel operator; that is
$$
R_{p_0,p_1,\fii}f(t) =\int_0^\infty k(t,r) f(r)dr, \qquad t > 0,
$$
where the kernel is
\begin{equation}\label{main_theorem:kernel_expresion}
k(t, r)= \fii \left(1-\log \frac r t \right)\left(\frac rt\right)^{\frac1{p_0}} \chi_{[0,t)}(r)  \frac1r + \left(\frac r{t}\right)^{\frac1{p_1}} \chi_{[t,\infty)} (r) \frac 1r.    
\end{equation}

\medskip

\noi By virtue of \cite[Theorem 3.3]{cs:cs}, the norm $\norm{R_{p_0,p_1,\varphi}}_{L^{p,1}(\mu)\rightarrow L^{p,\infty}}$ can be estimated by 
$$
A_k:=\sup_{t>0} \left( \sup_{s>0} \left(\frac ts\right)^\frac1p \int_0^s k(t,r) dr  \right).
$$

Now observe that for $\beta_0 = \max(1,\beta)$ and for every $0 < \alpha \leq 1$, by means of \eqref{por la logconcavidad2}, $$\varphi\left( 1 - \log x \right) \leq \varphi\left( \frac {\beta_0} \alpha x^{-\frac \alpha {\beta_0}}\right) \leq \varphi\left( \frac {\beta_0} \alpha \right)x^{-\frac {\alpha \beta} {\beta_0}} \leq \beta_0^{\beta_0} \varphi\left( \frac 1 \alpha \right)x^{-\alpha} , \qquad 0 < x \leq 1.$$
Take $\alpha = \frac 1{p_0} - \frac 1p \in (0,1)$. Hence, if $0 < s \leq t$,
\begin{align*}
\int_0^s k(t,r) \, dr= \frac1{t^\frac 1{p_0}}\int_0^s \fii \left(1-\log \frac r t \right)\frac{dr}{r^{1 - \frac1{p_0}}} \leq p_1\beta_0^{\beta_0}  \fii \left(\left[ \frac 1{p_0} - \frac 1p \right]^{-1} \right)\left( \frac st \right)^\frac 1p,
\end{align*}
while if $s>t$, we obtain
\begin{align*}
\int_0^s k(t,r)\, dr &= \frac1{t^\frac 1{p_0}}\int_0^t \fii \left(1-\log \frac r t \right)\frac{dr}{r^{1 - \frac1{p_0}}} + \frac1{t^\frac 1{p_1}}\int_t^s \frac{dr}{r^{1 - \frac1{p_1}}} \\ &\leq p_1\beta_0^{\beta_0}  \fii \left(\left[ \frac 1{p_0} - \frac 1p \right]^{-1} \right) + p_1\left( \frac st \right)^\frac 1p. 
\end{align*}

In consequence, we have that
$$
A_k\leq p_1\beta_0^{\beta_0}  \fii \left(\left[ \frac 1{p_0} - \frac 1p \right]^{-1} \right) \sup_{t>0} \left( \sup_{s>t} \left(\frac ts\right)^\frac1p\left[1 + \left(\frac st\right)^\frac1p\right]\right) = 2p_1\beta_0^{\beta_0}  \fii \left(\left[ \frac 1{p_0} - \frac 1p \right]^{-1} \right).
$$

\end{proof}

\begin{remark}
\noindent It is worth mentioning that in order to prove that \eqref{main_theorem:hypothesis_eq} implies \eqref{main_theorem:hypothesis_eq2}, the only properties that we have used of $\varphi$ are that $\fii$ is a nondecreasing function such that $\fii(1) = 1$ and that for every constant $C \geq 1$, $\fii(Cx) \approx \fii(x)$.
\end{remark}

As we mentioned in the introduction, one application of these pointwise estimates is to deduce extensions of Yano's extrapolation results as the following corollary shows. First, for an arbitrary measure $\mu$, some exponent $1 \leq p < \infty$ and some admissible  function $\fii$, we define the function space $L^{p,1}\varphi( \log L)(\mu)$ as the set of $\mu$-measurable functions $f$ satisfying $$\lVert f \rVert_{L^{p,1}\varphi( \log L)(\mu)} = \int_0^\infty \varphi\left(1 + \log^+\frac 1 r\right)f_\mu^*(r)\frac{dr}{r^{1 - \frac 1p}}  < \infty.$$

\begin{corollary}

Let $T$ be a sublinear operator such that for some $1 < p_0 < p_1 < \infty$ and for every $p_0 < p \leq p_1$,
\begin{equation*}
T:L^{p, 1}(\mu)\longrightarrow  L^{p, \infty}(\nu), \qquad C\varphi\left(\left[\frac 1{p_0} - \frac 1p\right]^{-1}\right),
\end{equation*}
where $\varphi$ is an admissible function. Then, if $\nu$ is a finite measure, 
$$
T: L^{p_0,1}\varphi( \log L)(\mu) \rightarrow L^{p_0,\infty}(\nu), \qquad \frac C{p_0 - 1}.
$$

\end{corollary}

\begin{proof}

\noindent As a consequence of Theorem~\ref{thrm:main_result1}, 
$$
 (Tf)_\nu^*(t) \lesssim \frac 1{p_0 - 1}R_{p_0, p_1, \fii}(f_\mu^*)(t), \qquad 0<t< \nu(\mathcal N).
$$ Further, by means of \cite[Theorem 3.3]{cs:cs}, 
\begin{align*}
    & \norm{R_{p_0,p_1,\varphi}}_{L^{p_0,1}\fii(\log L)(\mu)\rightarrow L^{p_0,\infty}(0,\nu(\mathcal N)} \lesssim \sup_{0 < t < \nu(\mathcal N)} t^\frac 1{p_0}\left[ \sup_{s>0} \frac{\displaystyle \int_0^s k(t,r) dr}{\displaystyle \int_0^s\varphi\left(1 + \log^+ \frac 1r\right)\frac{dr}{r^{1 - \frac 1{p_0}}}}  \right],    
\end{align*}
with $k(t,r)$ as in \eqref{main_theorem:kernel_expresion}. Hence, if $0 < s \leq t$,  
\begin{align*}
    t^\frac 1{p_0}\left[ \frac{\displaystyle \int_0^s k(t,r) dr}{\displaystyle \int_0^s\varphi\left(1 + \log^+ \frac 1r\right)\frac{dr}{r^{1 - \frac 1{p_0}}}}  \right] &= \frac{\displaystyle \int_0^s\varphi\left(1 + \log \frac tr\right)\frac{dr}{r^{1 - \frac 1{p_0}}}}{\displaystyle \int_0^s\varphi\left(1 + \log^+ \frac 1r\right)\frac{dr}{r^{1 - \frac 1{p_0}}}} \leq\max\left(1,\nu(\mathcal N)^\frac 1{p_0}\right), 
\end{align*}

\noindent while if $s > t$, we obtain \begin{align*}
    t^\frac 1{p_0}\left[ \frac{\displaystyle \int_0^s k(t,r) dr}{\displaystyle \int_0^s\varphi\left(1 + \log^+ \frac 1r\right)\frac{dr}{r^{1 - \frac 1{p_0}}}}  \right] &= \frac{\displaystyle t^\frac 1{p_0}\left(\frac 1{t^\frac 1{p_0}}\int_0^t\varphi\left(1 + \log \frac tr\right)\frac{dr}{r^{1 - \frac 1{p_0}}} + p_1 \left[\left(\frac s t \right)^\frac 1{p_1} - 1\right]\right)}{\displaystyle \int_0^s\varphi\left(1 + \log^+ \frac 1r\right)\frac{dr}{r^{1 - \frac 1{p_0}}}} \\ &\leq \frac{\displaystyle t^\frac 1{p_0}\left(C_\varphi + p_1 \left(\frac s t \right)^\frac 1{p_1} \right)}{\displaystyle p_0s^\frac 1{p_0}} \leq \frac{C_\varphi + p_1}{p_0},   
\end{align*}

\noindent so that $\norm{ R_{p_0,p_1,\varphi}}_{L^{p_0,1}\fii(\log L)(\mu)\rightarrow L^{p_0,\infty}((0,\nu(\mathcal N)),\, dx)} < \infty.$

\end{proof}

\begin{remark}

\noindent For $p_0 = 1$, we observe that following the lines of the sufficiency of the proof of Theorem~\ref{thrm:main_result1}, the only place where we could have problems is in \eqref{eq:Pg**mu}, since this estimate blows up as $p_0$ approximates $1^+$. Nevertheless, easy computations show that then, for every $t > 0$ and every $\mu$-measurable function $f$, $$(Tf)_\nu^*(t) \lesssim \frac 1t \int_0^t \left(1 - \log \frac rt\right)\varphi\left(1 - \log \frac rt\right)f_\mu^*(r)\,dr + \frac 1{t^\frac 1{p_1}}\int_t^\infty f_\mu^*(r)\frac{dr}{r^{1 - \frac 1{p_1}}}.$$ However, when $\varphi(x) = x^\alpha$, $\alpha > 0$, it can be deduced that for an arbitrary measure $\mu$ and a finite measure $\nu$, $$T:L(\log L)^{\alpha + 1}(\mu) \rightarrow L^{1,\infty}(\nu),$$ which, as we have seen on the introduction, is far from the best results known up to now (see, for instance, Theorem \ref{thrm:restricted_weak_type_Yano}).

\end{remark}

\bigskip

\centerline{\bf OPEN QUESTION}

\medskip

Can we extend our result to the case $p_0=1$ in an optimal way?

\end{document}